\documentclass{elsarticle}
\usepackage[english]{babel}
\usepackage{amsmath, amssymb, amsthm}
\usepackage{amsfonts, tikz}

\newtheorem{thm}{Theorem}[section]
\newtheorem{cor}{Corollary}[section]
\newtheorem{lem}{Lemma}[section]

\newtheorem{conj}{Conjecture}[section]
\theoremstyle{defn}

\theoremstyle{prb}
\newtheorem{prb}{Problem}[section]
\numberwithin{equation}{section}

\begin{document}

\begin{frontmatter}
	
\title{Non-divisibility of LCM Matrices by GCD Matrices on GCD-closed Sets}

\author[gazi]{Ercan Alt\i n\i \c{s}\i k \corref{corresponding}}
	\ead{ealtinisik@gazi.edu.tr}
\author[gazi]{Mehmet Y\i ld\i z }
	\ead{yildizm78@mynet.com}
\author[gazi]{Ali Keskin }
	\ead{akeskin1729@gmail.com}

\cortext[corresponding]{Corresponding author. Tel.: +90 312 202 1070}
\address[gazi]{Department of Mathematics, Faculty of Sciences, Gazi University \\ 06500 Teknikokullar - Ankara, Turkey}

\begin{abstract}
In this paper, we consider the divisibility problem of LCM matrices by GCD matrices in the ring $M_n(\mathbb{Z})$ proposed by Hong in 2002 and in particular a conjecture concerning the divisibility problem raised by Zhao in 2014. We present some certain gcd-closed sets on which the LCM matrix is not divisible by the GCD matrix in the ring $M_n(\mathbb{Z})$. This could be the first theoretical evidence that Zhao's conjecture might be true. Furthermore, we give the necessary and sufficient conditions on the gcd-closed set $S$ with $|S|\leq 8$ such that the GCD matrix divides the LCM matrix in the ring $M_n(\mathbb{Z})$ and hence we partially solve Hong's problem. Finally, we conclude with a new conjecture that can be thought as a generalization of Zhao's conjecture.
\end{abstract}


\end{frontmatter}

\section{Introduction}

Let $S=\{x_1,x_2,\ldots,x_n\}$ be a set of distinct positive integers and $f$ be an arithmetical function. We denote by $(f(S))$ and $(f[S])$ the $n \times n$ matrices on $S$ having $f$ evaluated at the greatest common divisor $(x_i,x_j)$ and the least common multiple $[x_i,x_j]$ of $x_i$ and $x_j$ as their $ij-$entries, respectively. If $f=I$, the identity function, the matrix $(I(S))$ is called the GCD matrix on $S$ and denoted by $(S)$. The LCM matrix $[S]$ is defined similarly. Given any positive real number $e$, let $\xi_e$ be the $e$-th power function. If $f=\xi_e$, then the matrices $(\xi_e(S))$ and $(\xi_e[S])$ are called the power GCD matrix and the power LCM matrix and we simply denote them by $(S^e)$ and $[S^e]$, respectively. In 1876, Smith \cite{Smith} proved that if $S=\left\lbrace 1,2,\ldots, n\right\rbrace $, then
$\det(S)=\prod_{k=1}^n (f\ast \mu)(k)$, where $f\ast \mu$ is the Dirichlet convolution of $f$ and the M\"{o}bius function $\mu$. Since then, many results on these matrices have been published in the
literature. For general accounts see e.g. \cite{Altinisik2005,AltinisikBuyukkose2015,BeslinLigh1989,BourqueLigh1992,HaukWangSillanpaa1997,HongLoewy2004,KorkeeHauk2003,Mattila2014,MattilaHauk2014}.

Let $A$ and $B$ be in $M_n(\mathbb{Z})$. We say that $A$ divides $B$ or $B$ is divisible by $A$ in the ring $M_n(\mathbb{Z})$ if there exists a matrix $C$ in $M_n(\mathbb{Z})$ such that $B=AC$ or $B=CA$, equivalently, $A^{-1}B \in M_n(\mathbb{Z})$ or $BA^{-1} \in M_n(\mathbb{Z})$. We simply write $A\mid B$ if $A$ divides $B$ in the ring $M_n(\mathbb{Z})$ and $A \nmid B$ otherwise. Divisibility is an interesting topic in the study of GCD and LCM matrices and the first result on the subject belongs to Bourque and Ligh. In 1992, they \cite{BourqueLigh1992} showed that if $S=\{x_1,x_2,\ldots,x_n\}$ is factor closed then $(S) \mid [S]$. A set $S$ is factor closed if it contains all divisors of $x$ for any $x$ in $S$. Then, in \cite{BourqueLigh1995}, they also proved that if $S$ is factor closed, $f$ is multiplicative and $(f\ast \mu)(x_i)\neq 0$ for all $x_i \in S$ then $(f(S)) \mid (f[S])$. A set $S$ is said to be gcd-closed if $(x_i,x_j)$ is in $S$ for all $1 \leq i,j \leq n$. Hong \cite{Hong2002} showed that for any  gcd-closed set $S$ with $|S| \leq 3$, $(S)\mid [S]$; however, for any integer $n \geq 4$, there is a gcd-closed set $S$ with $|S|=n$ such that $(S) \nmid [S]$. Along with the aforementioned results, Hong raised the following open problem in the same paper. 
\begin{prb}[\cite{Hong2002}] \label{HongProblem}
Let $n \geq 4$. Find necessary and sufficient conditions on the gcd-closed set $S$ with $|S|=n$ such that $(S) \mid [S]$.    
\end{prb}
Problem~\ref{HongProblem} was solved in particular cases $n=4$ and $n=5$ by Zhao \cite{Zhao2008} and Zhao-Zhao \cite{ZhaoZhao2009}, respectively. Providing a complete solution of Problem~\ref{HongProblem} is a hard task because there is no general method to construct all possible gcd-closed sets with $n$-elements. In \cite{Hong2002}, Hong used greatest-type divisors of the elements in $S$ to overcome this difficulty. Actually, the concept of greatest-type divisor was introduced by Hong in \cite{Hong1999JAlgebra} to prove the Bourque-Ligh conjecture \cite{BourqueLigh1992}. For $x,y\in S$ and $x<y$, if $x \mid y $ and the conditions $ x \mid z \mid y $ and $z \in S$ imply that $z \in \{ x,y \}$, then we say that $x$ is a greatest-type divisor of $y$ in $S$. For $x \in S$, we  denote by $G_S(x)$ the set of all greatest-type divisors of $x$ in $S$. In this frame, in \cite{Hong2006}, Hong conjectured that if $S$ is a gcd-closed set with $ \max_{x\in S} \{|G_S(x)|\}=1$, then $(S) \mid [S]$. Hong, Zhao and Yin \cite{HongZhaoYin2008ActaArith} proved Hong's conjecture and hence they solved Problem~\ref{HongProblem} for the particular case $ \max_{x\in S} \{|G_S(x)|\}=1$. Then, in \cite{FengHongZhao2009}, Feng, Hong and Zhao introduced a new method to investigate Problem~\ref{HongProblem} for the case $ \max_{x\in S} \{|G_S(x)|\}\leq 2$. They gave a new and elegant proof of Hong's conjecture. Let $e$ be a positive integer. Indeed, they proved that if $S$ is a gcd-closed set satisfying $ \max_{x\in S} \{|G_S(x)|\} \leq 2$, then  $(S^e) \mid [S^e]$ if and only if $ \max_{x\in S} \{|G_S(x)|\} =1$ or $ \max_{x\in S} \{|G_S(x)|\} = 2$ with $S$ satisfying the condition $\mathfrak{C}$. We say that an element $x\in S$ with $|G_S(x)|=2$ satisfies the condition $\mathfrak{C}$ if $[y_1,y_2]=x$ and $(y_1,y_2) \in G_S(y_1) \cap G_S(y_2)$, where $G_S(x)=\{y_1,y_2\}$.  We say that the set $S$ satisfies the condition $\mathfrak{C}$ if each element $x\in S$ with $|G_S(x)|=2$ satisfies the condition $\mathfrak{C}$. 

In addition to the aforementioned results, in \cite{HaukkanenKorkee2005}, Haukkanen and Korkee  investigated the divisibility of unitary analogues of GCD and LCM matrices in the ring $M_n(\mathbb{Z})$ and also, in \cite{KorkeeHaukkanen2008}, they considered Problem~\ref{HongProblem} for meet and join matrices when $n\leq 5$. On the other hand, Hong \cite{Hong2003} proved that $(f(S)) \mid (f[S])$ if $f$ is completely multiplicative and $S$ is a divisor chain or a multiple closed set, namely we have $y\in S$ if $ x \mid y \mid \text{lcm}(S)$ for any $x\in S$, where $\text{lcm}(S)$ denotes the least common multiple of all the elements in $S$.  Moreover, in a different point of view, many results on the divisibility of GCD and LCM matrices defined on particular sets have been published in the literature, see e.g. \cite{HeZhao2005,Hong2008,LiTan2011,LinTan2012,Tan2009,Tan2010,TanLi2013,TanLin2010,TanLinLiu2011,TanLuoLin2013,XuLi2011,ZhaoHongLiaoShum2007}. 

Recently, in \cite{Zhao2014}, Zhao solved Problem~\ref{HongProblem} when $5 \leq |S| \leq 7$. Indeed, he proved that $(S^e) \mid [S^e]$ if and only if $ \max_{x\in S} \{|G_S(x)|\} =1$, or $ \max_{x\in S} \{|G_S(x)|\} =2$ and $S$ satisfies the condition $\mathfrak{C}$. Thus, Problem~\ref{HongProblem} was solved for the case $|S| \leq 7$. In the same paper, Zhao raised the following conjecture.

\begin{conj}\label{Zhaoconj}
Let $S=\{x_1,x_2,\ldots x_n\}$ be a gcd-closed set with \newline $ \max_{x\in S} \{|G_S(x)|\} =m \geq 4$. If $n < \binom{m}{2} +m+2$ then $(S^e)\nmid [S^e]$. 
\end{conj}

Organization of the paper is as follows. In Section~2, we present some well-known lemmas such as Lemmas~\ref{BourqueLighLemma}, \ref{honglemma}, and \ref{FHZLEMMA} and some novel lemmas which concern the inverse of the GCD matrix on gcd-closed sets  and are important tools in the proof of our main results. In Section~3, firstly we give some results, in which we find some certain gcd-closed sets on which $(S)$ does not divide $[S]$. Secondly, using these results, which support the truth of Conjecture~\ref{Zhaoconj}, we give the necessary and sufficient conditions on the gcd-closed set $S$ with $|S|\leq 8$ such that $(S) \mid [S]$ in the ring $M_n(\mathbb{Z})$, and hence a particular solution to Problem~\ref{HongProblem} when $|S|\leq 8$. In the last section, we present a new conjecture that can be thought as a generalization of  Conjecture~\ref{Zhaoconj}.  

\section{Preliminaries}
We begin with a result of Bourque and Ligh \cite{BourqueLigh1992} providing a formula for the entries of the inverse of $(S^e)$ when $S$ is gcd-closed. Throughout this section, we always assume that $S=\{x_1,x_2,\ldots,x_n\}$ and $S$ is gcd-closed.
\begin{lem} [\cite{BourqueLigh1993}]	\label{BourqueLighLemma}
The inverse of the power GCD matrix $(S^e)$ on $S$ is the matrix $W=(w_{ij})$, where 
	$$w_{ij}=\sum_{x_i\mid x_k \atop x_j\mid x_k}\frac{c_{ik} c_{jk}}{\alpha_{e,k}}$$
	with 
	\begin{equation} \label{cij1}
	c_{ij}=\sum_{dx_i\mid x_j \atop dx_i\nmid x_t,x_t<x_j}^{}{\mu(d)}
	\end{equation}
	and
	\begin{equation}\label{alphaek}
	\alpha_{e,k}=\sum_{d\mid x_k \atop d\nmid x_t,x_t<x_k} (\xi_e* \mu )(d)
	\end{equation}  and $\xi_e(x)=x^e$.
\end{lem}
The following lemma, which was presented by Hong \cite{Hong2004JAlgebra}, provides a simple way to calculate $\alpha_{e,k}$, and the proof follows from the inclusion-exclusion principle. 

\begin{lem}[\cite{Hong2004JAlgebra}]\label{honglemma}
Let $G_S(x_k)=\{y_{k,1},\ldots,y_{k,m}\}$ be the set of the greatest type divisors of $x_k$ in $S$ $(1\leq k\leq n)$. Then 
	\begin{equation} \label{alphaekHong}
	\alpha _{e,k}=x_k^e+\sum_{t=1}^{m} (-1)^t \sum_{1\leq i_1< \cdots<i_t\leq m}^{}(x_k,y_{k,i_1},\ldots,y_{k,i_t})^e
	\end{equation} with $\alpha _{e,k}$ defined as in (\ref{alphaek}).
\end{lem}

Similarly, using the inclusion-exclusion principle, we obtain the following lemma for the values of $c_{ij}$ and $\alpha _{1,j}$. 

\begin{lem} \label{Lemmacijalphak} Let $G_S(x_j)=\{y_{j,1},\ldots,y_{j,m}\}$ be the set of the greatest type divisors of $x_j$ in $S$ $(1\leq j\leq n)$. Then
\begin{equation} \label{cij2}
	c_{ij}=\sum_{d\mid \frac{x_j}{x_i}} \mu (d)+\sum_{r=1}^{m}(-1)^r\sum_{1\leq i_1<\ldots <i_r\leq m}\sum_{d\mid \frac{(y_{j,i_1},\ldots ,y_{j,i_r})}{x_i}}\mu (d)
	\end{equation}
	and 
	\begin{equation} \label{alphajbycij}
	\alpha_j:=\alpha_{1,j}=\sum_{x_i\mid x_j} x_i c_{ij}.
	\end{equation}
\end{lem}
The values of $c_{ij}$ play an important role to determine the divisibility of LCM matrices by GCD matrices on gcd-closed sets. Therefore, we calculate the value of $c_{ij}$ in some particular cases. The first lemma belongs to Zhao \cite{Zhao2014}. 

\begin{lem}[\cite{Zhao2014}] \label{FHZLEMMA} 
If $x_i\in G_S(x_j)$ then $c_{ij}=-1$ and $c_{jj}=1$.
\end{lem}
Now, we introduce a new type subset of $S$.  Let $G_S(x_k)=\{y_{k,1},\ldots,y_{k,m}\}$ for $x_k\in S$.
We define $D_S(x_k)$ as follows:
$$D_S(x_k):=\{(y_{k,i_1},\ldots,y_{k,i_r}):2\leq r\leq m \text{ and } 1\leq i_1<\cdots<i_r\leq m \}.$$ 
In other words,  $D_S(x_k)$ is the set of all possible greatest common divisors of different greatest-type divisors of $x_k$.
Moreover, we recall the set $D_r=\{x\in S: x_r\mid x \text{ and } x>x_r\}$ for $x_r$ in $S$ which was defined by Feng, Hong and Zhao in \cite{FengHongZhao2009}. We give the second lemma for the value of $c_{ij}$, which is, in fact, a generalization of Lemma~2.7 in \cite{Zhao2014}. 
\begin{lem} \label{Dempty}
If $x_i\in D_S(x_j)$ and $D_S(x_j)\cap D_i=\emptyset ,$ then $c_{ij}=l_i-1$, where $l_i=|D_i\cap G_S(x_j)|$.
	
\end{lem}
\begin{proof}
	Let $G_S(x_j)=\{y_{j,1},\ldots ,y_{j,m}\}$. Since $x_i\in D_S(x_j)$ it is obvious that $D_i\cap G_S(x_j)\neq \emptyset$. Without loss of generality, we can assume that $D_i\cap G_S(x_j)=\{y_{j,1},\ldots ,y_{j,l_i}\}.$ Now, suppose that $G_S(x_j)-D_i=\emptyset$. Then, clearly $G_S(x_j)\subset D_i.$ Since $D_i\cap D_S(x_j)=\emptyset ,$ $D_S(x_j)$ must consist of only $x_i.$ In this case, it is clear that $|D_i\cap G_S(x_j)|=m$, and hence $l_i=m.$ Then, by (\ref{cij2}), we have 
	$$
	c_{ij}=(-1)^2\binom{m}{2} +(-1)^3\binom{m}{3} +\cdots +(-1)^m\binom{m}{m} =m-1.
	$$
	Now, consider the case $G_S(x_j)-D_i\neq \emptyset$.
	Let $y_{j,k}\in G_S(x_j)-D_i.$ If $y_{j,k}\in \{y_{j,i_1},\ldots ,y_{j,i_r}\}$ $(2\leq r\leq m)$, then, by the definition of $D_i$, we have $x_i\nmid y_{j,k}$, and hence $x_i\nmid (y_{j,i_1},\ldots ,y_{j,i_r})$. So, we have $\frac{(y_{j,i_1},\ldots ,y_{j,i_r})}{x_i}\notin \mathbb{Z}$.
	Then, we can write $c_{ij}$ as follows:
	\begin{equation} \label{cij3}
	c_{ij}=\sum_{d\mid \frac{x_j}{x_i}}\mu (d)+\sum_{r=1}^{l_i}(-1)^r\sum_{1\leq i_1<\ldots <i_r\leq l_i}\sum_{d\mid \frac{(y_{j,i_1},\ldots ,y_{j,i_r})}{x_i}}\mu (d).
	\end{equation}
	Since $x_i\in D_S(x_j)$, we have $x_i\notin G_S(x_j).$ So, by a well-known property of the M\"{o}bius function, $\sum_{d\mid \frac{y_{j,i_1}}{x_i}} \mu (d)=0$ for $1\leq i_1\leq l_i$. Thus, we can rewrite (\ref{cij3}) as follows: 
	\begin{equation} \label{cij4}
	c_{ij}=\sum_{r=2}^{l_i}(-1)^r\sum_{1\leq i_1<\ldots <i_r\leq l_i}\sum_{d\mid \frac{(y_{j,i_1},\ldots ,y_{j,i_r})}{x_i}}\mu (d).
	\end{equation}
	Since $D_i\cap D_S(x_j)=\emptyset$ and $\{y_{j,i_1},\ldots ,y_{j,i_r}\}\subset D_i,$ we have $(y_{j,i_1},\ldots ,y_{j,i_r})=x_i$ for every $2\leq r\leq l_i$ ($1\leq i_1<\cdots <i_r\leq l_i$). Then, by (\ref{cij4}), we have $c_{ij}=l_i-1.$
\end{proof}
When $x_i\in D_S(x_j)$ and $D_i\cap D_S(x_j)\neq \emptyset,$ it is really a hard task to calculate the values of $c_{ij}$ on all possible gcd-closed sets; however, by making some restrictions on the set $S$, we can obtain a formula for the values of $c_{ij}$. In order to do this, we denote by $\text{Min}(D_i\cap D_S(x_j))$ the set of all the minimal elements in $D_i\cap D_S(x_j)$ with respect to the divisibility relation on $S$.

\begin{lem} \label{Dnonempty}
	Let $x_i\in D_S(x_j)$, $D_i\cap D_S(x_j)\neq \emptyset$ and $\text{Min}(D_i\cap D_S(x_j))=\{x_{i,1},\ldots ,x_{i,k}\}$.  Let $|D_{i,r}\cap D_{i,t}\cap G_S(x_j)|\leq 1$ for all $1\leq r<t\leq k$ when $k\geq 2$. Then  $c_{ij}=l_i-\sum_{t=1}^{k}l_{i,t}+(k-1)$, where $l_{i,t}=|D_{i,t}\cap G_S(x_j)|$. 
\end{lem}
\begin{proof}
	Let $D_i\cap G_S(x_j)=\{y_{j,1},\ldots ,y_{j,l_i}\}$ without loss of generality. Since $x_i\in D_S(x_j)$, we can calculate $c_{ij}$ by (\ref{cij4}). Now, we consider the summand for $r=2$ in (\ref{cij4}). We want to find the number of terms such that $(y_{j,i_1},y_{j,i_2})=x_i$ or equivalently $(y_{j,i_1},y_{j,i_2})/x_i=1$ for $1\leq i_1<i_2\leq l_i.$ Since $|D_{i,r}\cap D_{i,t}\cap G_S(x_j)|\leq 1$,  $(y_{j,i_1},y_{j,i_2})$ is equal to $x_i$ or a multiple of only one element in $\text{Min}(D_i\cap D_S(x_j)).$ So, there exist $\binom{l_{i,1}}{2}$ terms such that $(y_{j,i_1},y_{j,i_2})$ is a multiple of $x_{i,1}$.
	By the same argument, the number of 2-tuples of $y_{j,i_1}$ and $y_{j,i_2}$ $(i_1<i_2)$ such that $(y_{j,i_1},y_{j,i_2})\neq x_i$ is $\sum_{t=1}^{k}\binom{l_{i,t}}{2}$.
	Here, it should be noted that there is no common subsets of $D_{i,r}\cap G_S(x_j)$ and $D_{i,t}\cap G_S(x_j)$ with two or more elements for $1\leq r<t\leq k$ by the hypothesis of the theorem.
	If we continue in this manner for $r=3,\ldots ,l_i$ we obtain that
	\begin{equation} \label{cij5}
	c_{ij}=(-1)^2\left[ \binom{l_i}{2}-\sum_{t=1}^{k}\binom{l_{i,t}}{2}\right] +\cdots +(-1)^{l_i}\left[ \binom{l_i}{l_i}-\sum_{t=1}^{k}\binom{l_{i,t}}{l_i}\right]. 
	\end{equation}
	Here, for convenience, we can assume $\binom{n}{m}=0$ whenever $n<m.$ Thus, we obtain
	\begin{eqnarray*}
		c_{ij} &=& \sum_{r=2}^{l_i} \left[ (-1)^r \left( \binom{l_i}{r}-\sum_{t=1}^{k}\binom{l_{i,t}}{r}\right) \right] \\
		&=& (l_i-1)- \sum_{t=1}^{k}(l_{i,t}-1) \\
		&=& l_i-\sum_{t=1}^{k}l_{i,t}+(k-1),
	\end{eqnarray*}
which concludes the proof. 
\end{proof}

Let $(L, \leq )$ be a finite meet semilattice. Haukkanen, Mattila and M\"{a}ntysalo determined the zeros of the M\"{o}bius function of $L$, see \cite[Lemma~3.1]{HaukkanenMattilaMantysalo2015}. If we take  $(L, \leq )=(S,|)$, where $S$ is a gcd-closed set of distinct positive integers and $|$ is the divisibility relation on $\mathbb{Z}$, we can restate their claim as follows:
\begin{equation*}
\mu_S(\frac{x}{z})=0 \text{  unless  } gcd(G_S(x))\ | \ z \ | \ x.
\end{equation*} 
The following lemma is a generalization of the above result in the number theoretical setting and by using it, we can determine the zeros of $c_{ij}$ on a gcd-closed set. 
\begin{lem} \label{cijzeroLemma}
	Let $1\leq i \neq j\leq n$. If $x_i \notin G_S(x_j)$ and  $x_i \notin D_S(x_j)$, then $c_{ij}=0$.
	\end{lem}

\begin{proof}
If $x_i \nmid x_j$, then it is clear that $c_{ij}=0$. Now, let $x_i \mid x_j$ and  $G_S(x_j)=\{y_{j,1},\ldots, y_{j,m}\}$. Since $\sum_{d\mid \frac{x_j}{x_i} } \mu (d)=0$ whenever $x_i \neq x_j$, by (\ref{cij2}), we have 
$$
c_{ij}=\sum_{r=1}^{m} (-1)^r \sum_{1\leq i_1 < \cdots < i_r \leq m} \sum_{ d\mid \frac{(y_{j,i_1},\ldots, y_{j,i_r})}{x_i}} \mu (d).
$$
Now, consider the sum $\sum_{ d\mid \frac{(y_{j,i_1},\ldots, y_{j,i_r})}{x_i}} \mu (d)$ for $1\leq r \leq m$. Since $x_i \notin G_S(x_j)$ and $x_i \notin D_S(x_j)$, we always have $(y_{j,i_1},\ldots, y_{j,i_r})\neq {x_i}$. Therefore, by a well-known property of the M\"{o}bius function, $\sum_{ d\mid \frac{(y_{j,i_1},\ldots, y_{j,i_r})}{x_i}} \mu (d)=0$ for all $1\leq r \leq m$. This completes the proof. 
\end{proof}

\begin{lem} \label{sumcijzeroLemma}
For $ j > 1$, we have $\sum_{i=1}^{n}c_{ij}=0$ or equivalently $\sum_{x_i\mid x_j}c_{ij}=0$.
\end{lem}
	
\begin{proof}
	Since $c_{jj}=1$ we have to prove that $\sum_{x_i\mid x_j \atop x_i < x_j} c_{ij}=-1$. Let $G_S(x_j)=\{ y_{j,1},\ldots, y_{j,m} \}$. Then, by (\ref{cij2}), we have
\begin{equation} \label{sumcij}
  \sum_{x_i\mid x_j \atop x_i < x_j }c_{ij} = \sum_{r=1}^{m} (-1)^r \sum_{1\leq i_1 < \cdots < i_r \leq m} \sum_{x_i\mid x_j \atop x_i<x_j} \sum_{d|\frac{(y_{j,i_1},\ldots, y_{j,i_r})}{x_i}} \mu(d).
\end{equation}
Let $k$ be an arbitrary fixed integer such that $1\leq k \leq m$. Now, we consider the summand for $r=k$ in (\ref{sumcij}). Since $\sum_{d|\frac{(y_{j,i_1},\ldots, y_{j,i_k})}{x_i}} \mu(d)=0$ unless $x_i=(y_{j,i_1},\ldots, y_{j,i_k})$, the summand for (\ref{sumcij}) is equal to $(-1)^k$ times the number of $k-$element subsets of $G_S(x_j)$, namely $(-1)^k\binom{m}{k}$. Thus, $\sum_{x_i\mid x_j \atop x_i < x_j}c_{ij} = \sum_{r=1}^{m} (-1)^r \binom{m}{r}=-1$. This completes the proof.   

\end{proof}

\begin{lem}\label{AliLemma}
	Let $G_S(x_n)=\{ y_{n,1},\ldots, y_{n,m} \}$ and $D_S(x_n)=\{ x_{n_1},\ldots, x_{n_t} \}$. Then, we have
	\begin{equation*}
	\left[ y_{n,1},\ldots, y_{n,m} \right] = \frac{\prod_{i=1}^{m}  y_{n,i}}{ \prod_{k=1}^t x_{n_k}^{c_{n_kn}} }.
	\end{equation*}
	with $c_{n_kn}$ defined as in (\ref{cij1}). Moreover, $\sum_{k=1}^{t} c_{n_kn}= m-1$. 
\end{lem}	

\begin{proof}
	Firstly we claim that for positive integers $a_1,a_2,\ldots ,a_m$,
	$$
	[a_1,\ldots ,a_m]=\frac{T_1 T_3\cdots T_{m-1}}{T_2 T_4\cdots T_{m}} 
	$$
	if $m$ is even
	$$
	[a_1,\ldots ,a_m]=\frac{T_1 T_3\cdots T_{m}}{T_2 T_4\cdots T_{m-1}} 
	$$
otherwise. Here $T_1=\prod_{i=1}^{m}a_i$ and $T_k=\prod_{1\leq i_1<\cdots <i_k\leq m} (a_{i_1},\ldots ,a_{i_k})$ for $2\leq k\leq m$.
We will prove the claim when $m$ is even. It is sufficient to prove that $[a_1,\ldots ,a_m]T_2\cdots T_m=T_1T_3\cdots T_{m-1}.$
	
Consider a prime number $p$ such that $p\mid [a_1,\ldots ,a_m]$. For $a\in \mathbb{Z}^+,$ let $\nu_p(a)$ denote the largest integer such that $p^{\nu_p(a)}$ divides $a$. Without loss of generality, we can assume that $\nu_p(a_1)\leq \cdots \leq \nu_p(a_m)$. Then, we have
\begin{eqnarray*}
		\nu_p\left( [a_1,\ldots ,a_m]T_2T_4\cdots T_m \right)  &=& \nu_p(a_m)+\sum_{i=1}^{m-1}\left (\sum_{j=1}^{m/2}\binom{m-i}{2j-1}\right )\nu_p(a_i)\\
		&=& \nu_p(a_m)+\sum_{i=1}^{m-1}\left (\sum_{j=1}^{m/2}\binom{m-i}{2j}\right )\nu_p(a_i)\\
		&=& \nu_p(T_1T_3\cdots T_{m-1}).
\end{eqnarray*}
Here, for convenience, we assume $\binom{i}{j}=0$ whenever $j>i$. Thus, 
$$
[a_1,\ldots ,a_m]T_2T_4\cdots T_m=T_1T_3\cdots T_{m-1}.
$$ 
We can similarly prove the case that $m$ is odd. Assuming that $m$ is even, by our claim, we have 
$$
	[y_{n,1},\ldots y_{n,m}]=\prod_{i=1}^{m}y_{n,i}\frac {T_3 T_5\cdots T_{m-1}}{T_2 T_4\cdots T_m}
$$
	where $T_k=\prod_{1\leq i_1<\cdots < i_k\leq m}(y_{n,i_1},\ldots ,y_{n,i_k})$.
	Since every $(y_{n,i_1},\ldots ,y_{n,i_k})$ is in $D_S(x_n)$ for $k \geq 2$, we can write $T_k=x_{n_1}^{\beta_{n_1}} \cdots x_{n_t}^{\beta_{n_t}}$, where each $\beta_{n_r}$ is a nonnegative integer for $1\leq r\leq t$. Indeed, 
$$\beta_{n_r}= \left| \left\lbrace (y_{n,i_1},\ldots y_{n,i_k}) :  (y_{n,i_1},\ldots y_{n,i_k})= x_{n_r}, 1\leq i_1<\cdots < i_k\leq m \right\rbrace \right|. $$ 
Thus, by (\ref{cij2}), it is clear that the exponent of $x_{n_r}$ in the fraction $\frac {T_2 T_4\cdots T_m}{T_3 T_5\cdots T_{m-1}}$	is equal to $c_{n_rn}$. Furthermore, by Lemma~\ref{sumcijzeroLemma}, we obtain that $\sum_{k=1}^{t} c_{n_kn}= m-1$.
\end{proof}

\section{Main Results}
In this section, we give main results of our paper. For the proof of the first three results, we use Zhao's approach \cite{Zhao2014}, that is, we will prove that an entry of the product $[S](S)^{-1}$ is in the interval $(0,1)$. Throughout this section, we denote $[S](S)^{-1}$ by $U$, where $[S]$ is the LCM matrix and $(S)$ is the GCD matrix, and we assume that $S$ is gcd-closed. 
\begin{thm}\label{DSxn1}
	Let $S=\{x_1,x_2,\ldots,x_m\}$ with $m>5$. Let $x_n\in S$ such that $n \geq 5$, $G_S(x_n)=\{x_2,\ldots,x_{n-1} \}$ and $gcd( G_S(x_n))=x_1$. If $x_i\mid x_n$ and $x_i\notin D_S(x_n)$ for all $n< i \leq m$, then $(S) \nmid [S]$.
\end{thm}

\begin{proof}
We have to prove that $U \notin M_m(\mathbb{Z})$. To perform this, it is sufficient to show that $U_{2n} \notin \mathbb{Z}$. By Lemmas~\ref{BourqueLighLemma} and \ref{FHZLEMMA}, we have 
$$
U_{2n}=\frac{x_n-\sum_{i=2}^{n-1}[x_2,x_i]+\sum_{i=n+1}^{m} [x_2,x_i]c_{in} + [x_2,x_1]c_{1n}}{\alpha_n}.
$$ 
By Lemma~\ref{cijzeroLemma}, we have $c_{in}=0$ for $n+1 \leq i \leq m$ since $x_i$ is neither in $G_S(x_n)$ nor in $D_S(x_n)$ whenever $i>n$. In addition to this, $c_{1n}=n-3$ by Lemma~\ref{Dempty}. Then, we have $U_{2n}=\frac{x_n-\sum_{i=2}^{n-1}[x_2,x_i]+x_2(n-3)}{\alpha_n}$. Also, by Lemma~\ref{Lemmacijalphak}, $\alpha_n=x_n - \sum_{i=2}^{n-1} x_i + x_1 (n-3)$. Letting 
$$\beta_n:= x_n-\sum_{i=2}^{n-1}[x_2,x_i]+x_2(n-3),$$ we can write $U_{2n}$ as $U_{2n}=\frac{\beta_n}{\alpha_n}$. Here, one can show that $\beta_n > 0$ and $\alpha_n > \beta_n$ using Zhao's approach as in the proof of Lemma~2.9 in \cite{Zhao2014}. So, we have $ 0 < U_{2n} < 1$ which means that $U_{2n} \notin \mathbb{Z}$.     
\end{proof}

\begin{thm}\label{DSxnchain}
	Let $S=\{x_1,x_2,\ldots,x_t\}$ with $t> 5$. Let $x_n\in S$ such that $n\geq 5$, $G_S(x_n)=\{x_2, \ldots,x_{n-1}\}$, $gcd(G_S(x_n))=x_1$ and $D_S(x_n)=\{x_1,x_{n+1},\ldots,x_m\}$ $(m<t)$. If $x_i\mid x_n$ for all $n< i \leq t$ and $D_S(x_n)$ is a divisor chain, then $(S)\nmid [S]$.	
\end{thm}
\begin{proof}
Since $D_S(x_n)$ is a divisor chain and $G_S(x_n)=\{x_2, \ldots,x_{n-1}\}$, we can assume that $x_1\mid x_{n+1} \mid x_{n+2} \mid \cdots \mid x_m$ and  $x_m \mid x_2$ without loss of generality. By Lemma~\ref{BourqueLighLemma}, we have 
$$
U_{2n}=\frac {\sum_{s=1}^{t}[x_2,x_s]c_{sn}}{\alpha_n}.
$$ 
By Lemmas \ref{FHZLEMMA}-\ref{cijzeroLemma}, we have 
	$$
	c_{sn}=\left\lbrace 
	\begin{array}{cl}
	(n-2)-l_{n+1}  &  \text{if }  s=1,\\ 
	-1 &   \text{if }  2\leq s\leq  n-1, \\
	1 &   \text{if }  s=n, \\ 
	l_s-l_{s+1} &  \text{if }  n+1\leq s\leq m-1, \\ 
	l_m-1 &  \text{if }  s=m, \\ 
	0  &  \text{if }  s>m.
	\end{array}
	\right. 
	$$
Then, we have $$ 
	U_{2n}=\frac{x_n-\sum_{i=2}^{n-1}[x_2,x_i] +\sum_{i=n+1}^{m-1}(l_i-l_{i+1}) x_2 +(l_m-1) x_2+[(n-2)-l_{n+1}]x_2}{\alpha_n},
	$$
and hence
	$$ U_{2n}=\frac{x_n-\sum_{i=2}^{n-1}[x_2,x_i]+x_2(n-3)}{\alpha_n}.$$
In what follows we let $\gamma_n:=x_n-\sum_{i=2}^{n-1}[x_2,x_i]+x_2(n-3)$.
	Since $n \geq 5$, we have  
$$ 
	\gamma_n>[x_2,\ldots,x_{n-1}]-\sum_{i=3}^{n-1}[x_2,x_i].
	$$
	By Lemma~\ref{AliLemma}, we know that  
\begin{equation*}
[x_2,\ldots,x_{n-1}]=\frac{\prod_{i=2}^{n-1}x_i}{x_1^{c_{1n} }  x_{n+1}^{c_{n+1,n}}\ldots x_m^{c_{mn}}},
\end{equation*} 
where $c_{1n}+c_{n+1,n}+\cdots+c_{mn}=n-3$ and $c_{in}>0$ for $i=1$ and $n+1\leq i\leq m$.
	
Suppose that $\max \{[x_2,x_i]:3\leq i\leq n-1\}=[x_2,x_r]$. By the definition of $D_S(x_n)$, it is clear that $(x_2,x_r) \in D_S(x_n)$. Without loss of generality, we can assume that $(x_2,x_r)=x_1$. Then
\begin{eqnarray*}
\gamma_n &>& \frac{x_2x_r}{x_1} \left(\frac{\prod_{i=3 \atop i\neq r}^{n-1}x_i}{x_1^{c_{1n} -1}  x_{n+1}^{c_{n+1,n}}\cdots x_m^{c_{nm}}}-(n-3) \right)\\ 
&\geq & \frac{x_2x_r}{x_1} \left( 2^{n-4}-(n-3)\right)\\
	&\geq & 0.
\end{eqnarray*}
On the other hand, by Lemma~\ref{Lemmacijalphak}, we have $$\alpha_n=x_n-\sum_{i=2}^{n-1}x_i +\sum_{i=n+1}^{m-1}(l_i-l_{i+1}) x_i +(l_m-1) x_m+[(n-2)-l_{n+1}]x_1.$$
Now, we show that $\alpha_n$ is greater than $\gamma_n$. 
\begin{eqnarray*}
\alpha_n - \gamma_n &=& \sum_{i=3}^{n-1}([x_2,x_i]-x_i) + (x_1-x_2)(n-2-l_{n+1})+(x_m-x_2)(l_m
-1)  \\
& & + \sum_{i=n+1}^{m-1}(x_i-x_2)(l_i-l_{i+1}).
\end{eqnarray*}
We claim that $| \{ x_i\in G_S(x_n): (x_2,x_i)=x_s\}|=c_{sn}$ for $s=1$ or $n+1\leq s \leq m$. For $s=1$, 
\begin{equation*}
\begin{array}{ll}
(x_2,x_i)=x_1 & \Leftrightarrow   x_{n+1} \nmid x_i \\
& \Leftrightarrow x_{i} \notin D_{n+1} \\
&  \Leftrightarrow x_{i} \in (G_S(x_n) \cap D_1) - (G_S(x_n) \cap D_{n+1})
\end{array}
\end{equation*}
and for $n+1 \leq s \leq m-1$,
\begin{eqnarray*}
(x_2,x_i)=x_s & \Leftrightarrow &  x_{s+1} \nmid x_i \\
& \Leftrightarrow & x_{i} \notin D_{s+1} \\
&  \Leftrightarrow &  x_{i} \in (G_S(x_n) \cap D_s) - (G_S(x_n) \cap D_{s+1}).
\end{eqnarray*}
Also, our claim for $s=m$ is a direct consequence of Lemma~\ref{Dempty}. 
Now, we can rewrite $\alpha_n -\gamma_n$ according to $(x_2,x_i)$ for $3\leq i \leq n-1$.
\begin{equation*}
\alpha_n -\gamma_n = \sum_{(x_2,x_i)=x_1} ( [x_2,x_i]-x_i+ x_1-x_2) + \sum_{k=n+1}^{m} \sum_{(x_2,x_i)=x_k} ( [x_2,x_i]-x_i + x_k-x_2).
\end{equation*} 
It is clear that in the first sum
\begin{equation*}
[x_2,x_i]-x_i+ x_1-x_2 = (\frac{x_2}{x_1}-1)(x_i-x_1) >0
\end{equation*}
and in the second sum
\begin{equation*}
[x_2,x_i]-x_i+ x_k-x_2 = (\frac{x_2}{x_k}-1)(x_i-x_k) >0.
\end{equation*}
Thus, $\alpha_n -\gamma_n >0$, and hence $U_{2n} = \frac{\gamma_n}{\alpha_n}$ is not an integer. 
\end{proof}

\begin{thm}\label{dsxn3}
	Let $S=\{ x_1, x_2, \ldots, x_m\}$ with $m> 5$. Let $x_n \in S$ such that  $G_S(x_n)=\{x_2,\ldots,x_{n-1} \}$, $gcd(G_S(x_n))=x_1$ and $D_S(x_n)=\{x_1,x_{n+1}, x_{n+2} \}$. If $x_i\mid x_n$ for all $n< i \leq m$, then $(S) \nmid [S]$. 
\end{thm}

\begin{proof}
	If $D_S(x_n)$ is a divisor chain then the proof is a direct consequence of Theorem~\ref{DSxnchain}. Now, let $D_S(x_n)$ be a $x_1-$set, namely $(x_{n+1}, x_{n+2})=x_1$. We will prove the claim of the theorem in two cases as the set $(G_S(x_n) \cap D_{n+1}) \cap (G_S(x_n) \cap D_{n+2})$ can be empty or a singleton subset of $G_S(x_n)$. 
	
	Now, let $G_S(x_n) \cap D_{n+1} \cap D_{n+2} \neq \emptyset $. The set $G_S(x_n) \cap D_{n+1} \cap D_{n+2}$ cannot have more than one element. Suppose the contrary, that is, $x_i,x_j \in G_S(x_n) \cap D_{n+1} \cap D_{n+2}$. Since $S$ is gcd-closed and $D_S(x_n)=\{x_1,x_{n+1}, x_{n+2} \}$, we have $(x_i,x_j)=x_{n+1}$ or $x_{n+2}$. Now, assume that $(x_i,x_j)=x_{n+1}$. On the other hand, $x_{n+2} \mid (x_i,x_j)$  since $x_i,x_j \in D_{n+2}$. Then, we have $x_{n+2} \mid x_{n+1}$, a contradiction. Thus, we can assume that $G_S(x_n) \cap D_{n+1} \cap D_{n+2}=\{x_2\}$ without loss of generality. We will show that $U_{2n} \notin \mathbb{Z}$. By Lemmas \ref{FHZLEMMA}-\ref{cijzeroLemma}, it is clear that 
	$$
	c_{sn}=\left\lbrace 
	\begin{array}{cl}
	n-l_{n+1}-l_{n+2}-1  &  \text{if }  s=1,\\ 
	-1 &   \text{if }  2\leq s\leq  n-1, \\ 
	1 &   \text{if }  s=n, \\ 
	l_{n+1}-1 &   \text{if }   s =n+1, \\
	l_{n+2}-1 &   \text{if }   s =n+2, \\ 
	0  &  \text{if }  s>n+2.
	\end{array}
	\right. 
	$$
Thus, by Lemma~\ref{BourqueLighLemma}, we have
\begin{eqnarray*}
U_{2n}&=&\frac{1}{\alpha_n} 
\bigg( x_n-\sum_{i=2}^{n-1} [x_2,x_i] + [x_2,x_{n+1}](l_{n+1}-1)+[x_2,x_{n+2}](l_{n+2}-1)  \\
& &  +[x_2,x_1] (n-l_{n+1}-l_{n+2}-1) \bigg) 
\end{eqnarray*}
Since $x_2$ is a multiple of $\text{lcm}(D_S(x_n))$, by Lemma~\ref{Lemmacijalphak}, we have
\begin{equation*}
U_{2n}= \frac{x_n-\sum_{i=2}^{n-1} [x_2,x_i] + x_2(n-3)}{\alpha_n}, 
\end{equation*}
where
\begin{equation*}
\alpha_n = x_n-\sum_{i=2}^{n-1} x_i + x_{n+1}(l_{n+1}-1)+x_{n+2}(l_{n+2}-1) + x_1 (n-l_{n+1}-l_{n+2}-1).
\end{equation*} 
Let 
$$
\gamma_n = x_n-\sum_{i=2}^{n-1} [x_2,x_i] + x_2(n-3).
$$ 
Using the same method as in the proof of Theorem~\ref{DSxnchain}, one can easily show that $\gamma_n $ is positive and $| \{ x_i\in G_S(x_n) : x_k=(x_2,x_i)\}|=c_{kn}$ for $k=1,n+1,n+2$. So, it is sufficient to show that $ \alpha_n - \gamma_n $ is positive. To do this, we write $\alpha_n - \gamma_n$ as follows:
\begin{eqnarray*}
\alpha_n - \gamma_n &=& \sum_{  (x_2,x_i)=x_{n+1} \atop x_i\in G_S(x_n)} ([x_2,x_i]-x_i+(x_{n+1}-x_2)) \\ 
& & +\sum_{(x_2,x_i)=x_{n+2} \atop x_i\in G_S(x_n)} ([x_2,x_i]-x_i+(x_{n+2}-x_2)) \\
& & +\sum_{(x_2,x_i)=x_{1} \atop x_i\in G_S(x_n)} ([x_2,x_i]-x_i+(x_{1}-x_2)) \\
&=& \sum_{  (x_2,x_i)=x_{n+1} \atop x_i\in G_S(x_n)} (\frac{x_2}{x_{n+1}}-1)(x_i-x_{n+1}) \\ 
& & +\sum_{(x_2,x_i)=x_{n+2} \atop x_i\in G_S(x_n)} (\frac{x_2}{x_{n+2}}-1)(x_i-x_{n+2}) \\
& & +\sum_{(x_2,x_i)=x_{1} \atop x_i\in G_S(x_n)} (\frac{x_2}{x_{1}}-1)(x_i-x_{1}). 
\end{eqnarray*}
Then, it is clear that $\alpha_n -\gamma_n > 0$. 

Now, we investigate the case $[D_{n+1}\cap D_{n+2}]\cap G_S(x_n)=\emptyset$. Without loss of generality, we can assume that $D_{n+1}\cap G_S(x_n)=\{x_2,\ldots , x_{k+1}\}$ and $D_{n+2}\cap G_S(x_n)=\{x_{k+2},\ldots , x_{k+s+1}\}$. In this case, by Lemmas~\ref{BourqueLighLemma},~\ref{FHZLEMMA} - \ref{cijzeroLemma}, we have
\begin{eqnarray*}
U_{2n}&=& \frac{1}{\alpha_n}  \bigg(x_n-\sum_{i=2}^{n-1} [x_2,x_i] + [x_2,x_{n+1}](k-1)+[x_2,x_{n+2}](s-1) \\
& & +[x_2,x_{1}](n-k-s-1)\bigg).
\end{eqnarray*}
Also, by Lemma~\ref{Lemmacijalphak},
$$
\alpha_n =x_n-\sum_{i=2}^{n-1} x_i + (k-1)(x_{n+1})+(s-1)x_{n+2}+(n-k-s-1)x_1.
$$
Let 
\begin{equation*}
\gamma_n := x_n-\sum_{i=2}^{n-1} [x_2,x_i] + [x_2,x_{n+1}](k-1)+[x_2,x_{n+2}](s-1)+[x_2,x_{1}](n-k-s-1).
\end{equation*}
Using a similar method as in the proof of Theorem~\ref{DSxnchain}, one can show that $\gamma_n > 0$. Now, we will prove that $\alpha_n-\gamma_n$ is positive. 
 
\begin{eqnarray*}
\alpha_n-\gamma_n & = &
(k-1)(x_{n+1}-x_2)+\sum_{i=3}^{k+1}\left( [x_2,x_i]-x_i\right) \\
&& +(s-1)(x_{n+2}-[ x_2,x_{n+2}])+\sum_{i=k+2}^{k+s+1}\left( [x_2,x_i]-x_i\right) \\
&& +(n-k-s-1)(x_1-x_2)+\sum_{i=k+s+2}^{n-1}\left( [x_2,x_i]-x_i\right)\\
& = & \sum_{i=3}^{k+1}\left( [x_2,x_i]-x_i+(x_{n+1}-x_2)\right)\\
& & +\sum_{i=k+2}^{k+s+1}\left( [x_2,x_i]-x_i+(x_{n+2}-x_2)\right) \\
& & +\sum_{i=k+s+2}^{n-1}\left( [x_2,x_i]-x_i+(x_1-x_2)\right)\\
& & +[x_2,x_{n+2}]-x_{n+2}+x_1-x_2\\
&=& \sum_{  i=3}^{k+1} (\frac{x_2}{x_{n+1}}-1)(x_i-x_{n+1}) +\sum_{  i=k+2}^{k+s+1} (\frac{x_2}{x_{1}}-1)(x_i-x_{n+2}) \\
& & +\sum_{  i=k+s+2}^{n-1} (\frac{x_2}{x_{1}}-1)(x_i-x_{1}) + (\frac{x_2}{x_{1}}-1)(x_{n+2}-x_{1})\\
& >& 0.
\end{eqnarray*}
This completes the proof.
\end{proof}
After the proof of Theorems~\ref{DSxn1}-\ref{dsxn3}, we can say that Zhao's approach works when $x_n$ is a maximal element of $S$ with respect to the divisibility relation. Does the same method work if $S$ contains some multiples of $x_n$? It appears to be difficult to answer this question without the following lemma.

\begin{lem} \label{toplemma}
	Let $S=\{x_1,x_2,\ldots ,x_m\}$ such that $i\leq j$ whenever $x_i\mid x_j$. Also, let $x_n \in S$ and $D_n\cup \{x_n\}=\{x_n=x_{n_1},\ldots ,x_{n_t}\}$. Then, for each $1\leq q \leq m $, 
	$$
	\sum_{i=1}^{t}U_{qn_i}=\sum_{s=1}^{m}[x_q,x_s] \frac{c_{sn}}{\alpha_n}.
	$$
\end{lem}
\begin{proof}
	By Lemma \ref{BourqueLighLemma}, we have
	\begin{eqnarray*}
		\sum_{i=1}^{t}U_{qn_i} &=& \sum_{i=1}^{t}\left[\sum_{s=1}^{m}[x_q,x_s]\sum_{x_s\mid x_k \atop x_{n_i}\mid x_k}c_{sk}\frac{c_{n_ik}}{\alpha_k}\right] \\
		&=& \sum_{s=1}^{m}[x_q,x_s]\sum_{i=1}^{t}\sum_{k=n}^{m} c_{sk} \frac{c_{n_ik}}{\alpha_k} \\
		&=& \sum_{k=n}^{m}\sum_{s=1}^{m}[x_q,x_s]\frac{c_{sk}}{\alpha_k}\sum_{i=1}^{t}c_{n_ik} \\
		&=& \sum_{s=1}^{m}[x_q,x_s]\frac{c_{sn}}{\alpha_n}\sum_{i=1}^{t}c_{n_in}+\sum_{k=n+1}^{m}\sum_{s=1}^{m}[x_q,x_s]\frac{c_{sk}}{\alpha_k}\sum_{i=1}^{t}c_{n_ik} 
	\end{eqnarray*}
	Here $\sum_{i=1}^{t}c_{n_in}=c_{nn}=1$ and $\sum_{i=1}^{t}c_{n_ik}=\sum_{x_{n_i}\mid x_k}c_{n_ik}$.
	The last sum is over $x_{n_i}\in D_n\cup \{x_n\}$ dividing the fixed $x_k\in D_n.$
	Since $S$ is gcd-closed, $D_n\cup \{x_n\}$ is also gcd-closed.
	Now, let $(c_{ij})_A$ denote $c_{ij}$ for a gcd-closed set $A$, as defined in Lemma~\ref{BourqueLighLemma}.
	We want to show that $(c_{n_ik})_S=(c_{n_ik})_{D_n\cup \{x_n\}}.$
	Let $G_S(x_{n_j})=\{y_{j,1},\ldots ,y_{j,k}\}$. By Lemma~\ref{Lemmacijalphak},
	
	\begin{equation} \label{cijdifsets}
		(c_{n_in_j})_S=\sum_{d\mid \frac{x_{n_j}}{x_{n_i}}}\mu (d)+\sum_{r=1}^{k}(-1)^r\sum_{1\leq i_1<\ldots <i_r\leq k}\sum_{d\mid \frac{(y_{j,i_1},\ldots ,y_{j,i_r})}{x_{n_i}}}\mu (d).	
	\end{equation}
	Without loss of generality, let $x_{n_i}\nmid y_{j,k}$. Then $x_{n_i}\nmid (y_{j,i_1},\ldots ,y_{j,i_{r-1}},y_{j,k})$, and hence $(y_{j,i_1},\ldots ,y_{j,i_{r-1}},y_{j,k})/x_{n_i}\notin \mathbb{Z}.$
	So, if $y_{j,k}\in \{y_{j,i_1},\ldots ,y_{j,i_r}\}$,  then the summation  $\sum_{d\mid (y_{j,i_1},\ldots ,y_{j,i_r})/x_{n_i}}\mu (d)$ is empty, and hence it is equal to zero.
	Thus, letting $G_S(x_{n_j})\cap (D_{n_i}\cup \{x_{n_i}\})=\{y_{j,1},\ldots ,y_{j,u}\}$
	without loss of generality, we can write (\ref{cijdifsets}) as follows
	\begin{equation*}
		(c_{n_in_j})_S=\sum_{d\mid \frac{x_{n_j}}{x_{n_i}}}\mu (d)+\sum_{r=1}^{u}(-1)^r\sum_{1\leq i_1<\ldots <i_r\leq u}\sum_{d\mid \frac{(y_{j,i_1},\ldots ,y_{j,i_r})}{x_{n_i}}}\mu (d).
	\end{equation*}
	On the other hand, it is clear that $G_S(x_{n_j})\cap (D_{n_i}\cup \{x_{n_i}\})\subset G_S(x_{n_j})\cap (D_{n}\cup \{x_{n}\})$ and $G_{D_{n}\cup \{x_{n}\}}(x_{n_j})=G_S(x_{n_j})\cap (D_{n}\cup \{x_{n}\}).$
	Thus, we obtain $(c_{n_in_j})_S=(c_{n_in_j})_{D_{n}\cup \{x_{n}\}}.$

	Now, since $D_{n}\cup \{x_{n}\}$ is a gcd-closed set and $(c_{n_ik})_{D_{n}\cup \{x_{n}\}}=(c_{n_ik})_S$, we have $\sum_{i=1}^{t}c_{n_ik}=0$ by Lemma~\ref{sumcijzeroLemma} for $x_k\in D_n.$
	Thus, $$
	\sum_{i=1}^{t}U_{qn_i}=\sum_{s=1}^{m}[x_q,x_s] \frac{c_{sn}}{\alpha_n}.
	$$
\end{proof}

Putting Theorems~\ref{DSxn1}-\ref{dsxn3} and Lemma~\ref{toplemma} together, we have the following result.
\begin{thm} \label{maintheorem}
	Let $S=\{x_1,x_2,\ldots ,x_m\}$ and let $S$ have an element $x$ with $|G_S(x)|\geq 3$. If $D_S(x)$ is a divisor chain or $|D_S(x)|\leq 3$, then $(S)\nmid [S]$.
\end{thm}
\begin{proof}
	Without loss of generality, we can assume that $x_n \in S$ such that $5\leq n \leq m$, $ G_S(x_n)=\{x_2,\ldots ,x_{n-1}\}$, and $\text{gcd}(G_S(x_n))=x_1$. Also, let $D_{n}\cup \{x_{n}\}=\{x_n=x_{n_1},\ldots x_{n_t}\}$. By Lemma~\ref{toplemma}, we have
	$$
	\sum_{i=1}^{t}U_{2n_i}=\sum_{s=1}^{m}[x_2,x_s]\frac{c_{sn}}{\alpha_n}.
	$$ 
We have two cases that $D_S(x_n)$ could be a divisor chain or not. In both cases, one can show that $\sum_{s=1}^{m}[x_2,x_s].c_{sn}/\alpha_n\notin \mathbb{Z}$ by similar methods to the proofs of Theorems~\ref{DSxnchain} and \ref{dsxn3}, respectively. 
\end{proof}

So far, we have proven that if $S$ has an element $x$ such that $|G_S(x)|\geq 3$, and $|D_S(x)|\leq 3$ or $D_S(x)$ is a divisor chain, then the divisibility does not hold. On the other hand, for the complete solution of  Problem~\ref{HongProblem} for $|S| \leq 8$, whether the divisibility holds when $S$ has an element $x$ such that $|G_S(x)|=3$ and $|D_S(x)| = 4 $ remains unsolved. The following condition is a key to the divisibility for this case. For $x\in S$, we say that $x$ satisfies the condition $\mathfrak{M}$ if $[x_{i},x_{j}]=x$ for all different $x_{i},x_{j}\in G_S(x)$ when $|G_S(x)|\geq 2$. Also, we say that the set $S$ satisfies the condition $\mathfrak{M}$ if each element $x \in S$ with $|G_S(x) | \geq 2$ satisfies the condition $\mathfrak{M}$. Recall that the condition $\mathfrak{C}$ is defined for the elements with only two greatest-type divisors. If $x\in S$ satisfies the condition $\mathfrak{C}$, then it clearly satisfies the condition $\mathfrak{M}$. On the other hand, an element satisfying the condition $\mathfrak{M}$ need not satisfy the condition $\mathfrak{C}$. 

\begin{thm} \label{G3D4}
Let $S=\{x_1,x_2,\ldots ,x_8\}$, $|G_S(x_8)|=3$ and $|D_S(x_8)|=4$. Then, $(S)\mid [S]$ if and only if $S$ satisfies the condition $\mathfrak{M}$.  
\end{thm}
\begin{proof}

Under the hypothesis of the theorem we can assume that the Hasse diagram of $S$ with respect to the divisibility relation is as follows:
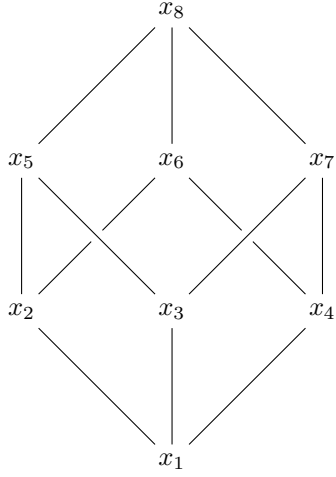
\begin{center}
\begin{tikzpicture}
\node (max) at (0,4) {$x_8$};
\node (a) at (-2,2) {$x_5$};
\node (b) at (0,2) {$x_6$};
\node (c) at (2,2) {$x_7$};
\node (d) at (-2,0) {$x_2$};
\node (e) at (0,0) {$x_3$};
\node (f) at (2,0) {$x_4$};
\node (min) at (0,-2) {$x_1$};
\draw (min) -- (d) -- (a) -- (max) -- (b) -- (f)
(e) -- (min) -- (f) -- (c) -- (max)
(d) -- (b);
\draw[preaction={draw=white, -,line width=6pt}] (a) -- (e) -- (c);
\end{tikzpicture}
\end{center}
If $[x_{k,i},x_{k,j}]=x_k$ for all different $x_{k,i},x_{k,j}\in G_S(x_k)$ when $|G_S(x_k)|\geq 2$, then by a direct computation, one can obtain that 
\begin{equation*}
U_{ij} = \left\lbrace \begin{array}{cl}
\frac{x_i}{x_1} & \text{if } [x_i,x_j]=x_8 \text{ and } (x_i,x_j)=x_1,\\
0 & otherwise.
\end{array}\right. 
\end{equation*}

We will show non-divisibility of the LCM matrix by the GCD matrix on $S$ in two cases. 

Case~1. Let $S$ have an element $x_k$ such that $G_S(x_k)=\left\lbrace x_{k,1},x_{k,2}\right\rbrace$ and $[x_{k,1},x_{k,2}] < x_k$. Without loss of generality, we can take $x_k = x_5$. Then, it is clear that $[x_2,x_3]<x_5$. By Lemmas~\ref{BourqueLighLemma}, \ref{FHZLEMMA}-\ref{Dnonempty}, and \ref{toplemma}, we have
\begin{eqnarray*}
U_{25} + U_{28} &=& \frac{\sum_{s=1}^{8} [x_2,x_s]c_{s5}}{\alpha_5} \\
&=& \frac{x_5-[x_2,x_3]}{\alpha_5}.
\end{eqnarray*}
By Lemma~\ref{Lemmacijalphak}, we have $\alpha_5 = x_5-x_2-x_3+x_1$. Since $[x_2,x_3]<x_5$, we have $x_5-[x_2,x_3]>0$ and 

\begin{equation*}
\alpha_5 - \left( x_5-[x_2,x_3] \right) = \left( \frac{x_3}{x_1} -1 \right)\left(x_2-x_1 \right)  >0.
\end{equation*}
Thus, $0< U_{25} + U_{28} <1$. That is $U \notin M_8(\mathbb{Z})$. 

Case~2. Let $[x_{5},x_{6}] < x_8$ without loss of generality. Now, we must have $[x_{2},x_{3}]=x_5$, $[x_{2},x_{4}]=x_6$ and $[x_{3},x_{4}]=x_7$ otherwise the proof is obvious by Case~1. Under these assumptions, we have 
\begin{equation*}
[x_5,x_6,x_7]=[x_5,x_6] < x_8. 
\end{equation*}
We will show that $U_{58} \notin \mathbb{Z}$. By Lemmas~\ref{BourqueLighLemma}, \ref{FHZLEMMA}-\ref{Dnonempty}, we have
\begin{equation*}
U_{58} = \frac{x_8-[x_5,x_6]-[x_5,x_7]+[x_4,x_5]}{\alpha_8}.
\end{equation*}
Let $\gamma_8 = x_8-[x_5,x_6]-[x_5,x_7]+[x_4,x_5]$. Since $[x_5,x_6,x_7] < x_8$ and clearly $[x_5,x_6,x_7] \mid  x_8$, we have $x_8 \geq 2\cdot [x_5,x_6,x_7]$, and hence 
 \begin{equation*}
 \gamma_8 > x_8-[x_5,x_6]-[x_5,x_7] \geq 2\cdot [x_5,x_6,x_7]- [x_5,x_6]-[x_5,x_7] \geq 0.
 \end{equation*}
By Lemma~\ref{Lemmacijalphak}, we have $\alpha_8= x_8 -x_5-x_6-x_7+x_2+x_3+x_4-x_1$. Then,
\begin{eqnarray*}
\alpha_8 - \gamma_8 &=& \left( [x_5,x_6] - x_6 -x_5 +x_2 \right)  + \left( [x_5,x_7] - x_7 -x_5 + x_3 \right) \\
& & + \left( -[x_4,x_5] - x_1 +x_5 + x_4 \right) \\
&=& \left(  \frac{x_5}{x_1} - 1 \right) \left(  x_6 - x_2 \right)   + \left(  \frac{x_5}{x_3} - 1 \right) \left(  x_7 - x_3 \right) + \left(  \frac{x_4}{x_1} - 1 \right) \left(  x_1 - x_5 \right) \\
	&=& \left(  \frac{x_4}{x_1} - 1 \right) \left(  x_5-x_2+x_1-x_3 \right) \\
	&=& \left(  \frac{x_4}{x_1} - 1 \right) \left(  \frac{x_3}{x_1} - 1 \right) (x_2-x_1) >0.
\end{eqnarray*}
This completes the proof.
\end{proof}

\begin{cor}
	Let $S$ be a gcd-closed set with $\left| S\right|\leq 8 $.
	$(S)\mid [S]$ if and only if 
\item[i)] $\max_{x\in S}\{\left| G_S(x)\right| \}=1$  or 
	
\item[ii)] $\max_{x\in S}\{\left| G_S(x)\right| \}=2$ and $S$ satisfies the condition $\mathfrak{C}$   or
	
\item[iii)] $\max_{x\in S}\{\left| G_S(x)\right| \}=3$ and  $S$ satisfies the condition $\mathfrak{M}$.  
\end{cor}
\begin{proof}
If (i) or (ii) holds then by Theorems~3.4 and 4.7 in \cite{FengHongZhao2009} we know $(S)\mid [S].$
Now, let (iii) hold. Let $|S|=n$ with $n\leq 8$, $G_S(x_n)=\left\lbrace x_{n,1},x_{n,2},x_{n,3}  \right\rbrace $ and let $S$ satisfy the condition $\mathfrak{M}$. Then we claim that $|D_S(x_n)|= 4$ and $|S|=8$. Since $S$ satisfies the condition $\mathfrak{M}$, we must have   
$$
[x_{n,1},x_{n,2}]=[x_{n,1},x_{n,3}]=[x_{n,2},x_{n,3}]
$$
and hence $(x_{n,1},x_{n,2})$, $(x_{n,1},x_{n,3})$ and $(x_{n,2},x_{n,3})$ must be different elements in $S$. This means that $|D_S(x_n)|= 4$, and hence $|S|=8$. So, we must investigate the case that $|S|=8$ and $|D_S(x_8)|= 4$. Thus, by Theorem~\ref{G3D4}, we have $(S)\mid [S]$. 

Now, we prove the necessary part of the theorem by contrapositive. If $\max_{x\in S}\{\left| G_S(x)\right| \}=2$ and $S$ does not satisfy the condition $\mathfrak{C}$, then by Theorems~4.7 in \cite{FengHongZhao2009}, we know $(S)\nmid [S]$.
Consider the case that $\max_{x\in S}\{\left| G_S(x)\right| \}=3$ and  $S$ does not satisfy the condition $\mathfrak{M}$.
If  $|D_S(x)|\leq 3$ for the element $x$ with three greatest-type divisors, then we have $(S)\nmid [S]$ by Theorem~\ref{maintheorem}. If $|D_S(x)|=4$, then we have $(S)\nmid [S]$ by  Theorem~\ref{G3D4}. Since $|S|\leq 8$, $D_S(x)\leq 3$ if $\max_{x\in S}\{\left| G_{S}(x)\right| \} \geq 4$. Thus, by Theorem~\ref{maintheorem}, we have $(S) \nmid [S]$. This completes the proof of the necessary part.         
\end{proof}
Let $e\geq 1$ be an integer. All the results that we have obtained in this section are valid for the $e$th power GCD matrix and the $e$th power LCM matrix. In this paper, we have only  considered the original version of Problem~\ref{HongProblem} for the sake of brevity. 

\section{A new conjecture}
Let $k$ and $i$ be arbitrary positive integers. Consider the set 
$$
S_i=\left\lbrace p,p^2,\ldots,p^k,p^kq_1,p^kq_2,\ldots,p^kq_i,p^kq_1 q_2\ldots q_i \right\rbrace, 
$$ 
where $q_1,\ldots,q_i$ and $p$ are different prime numbers. It is clear that \newline $\max_{x\in S_i} \{\left| G_{S_i}(x)\right|\}=i$ and $\left| D_{S_i}(p^kq_1\ldots q_i)\right| =1$.
If $i\geq 3$, then, by a direct consequence of Theorem~\ref{maintheorem}, we have $(S_i)\nmid [S_i]$.
Let $k=10$ and $i=4$. We have $\max_{x\in S_4}\{\left| G_{S_4}(x)\right| \}=4$ and $|S_4|=15$. Thus, we have a gcd-closed set, not satisfying the hypothesis of Conjecture~\ref{Zhaoconj}, but the divisibility for this set cannot hold. Moreover, $S_4$ does not satisfy the condition $\mathfrak{M}$. Therefore, in the light of our results, we can say that the non-divisibility depends on not only the number $\max_{x\in S}\{\left| G_S(x)\right| \}$ but also the condition $\mathfrak{M}$. Indeed, a reason preventing the divisibility is that $S$ does not satisfy the condition $\mathfrak{M}$. 

If a set $S$ satisfies the hypothesis of Zhao's conjecture, then there must be at least three elements $x_{m,i_1}$, $x_{m,i_2}$, and $x_{m,i_3}$ such that $(x_{m,i_1},x_{m,i_2})=(x_{m,i_1},x_{m,i_3})=(x_{m,i_2},x_{m,i_3})$ where $|G_S(x_m)|=m$ and $x_{m,i_k}\in G_S(x_m)$ for $1\leq k\leq 3$. Then, we have $[x_{m,i_1},x_{m,i_2}] < x_m$. This means that if the set $S$ with $|S|=n$ holds  the hypothesis of Zhao's conjecture, then $S$ does not satisfy the condition $\mathfrak{M}.$ 

Finally, after the above observations, we conclude our paper with a new conjecture, which is a generalization of Conjecture~\ref{Zhaoconj}.

\begin{conj}
Let $S$ be a gcd-closed set with $\max_{x\in S}\{\left| G_S(x)\right| \} \geq 2$. If $S$  does not satisfy the condition $\mathfrak{M}$, then $(S)\nmid [S]$. 
\end{conj}

\end{document}